\documentclass[10pt,a4paper]{article}
\usepackage[utf8]{inputenc}
\usepackage{amsmath,amssymb,amsthm,amsfonts,epsfig,enumerate}
\usepackage{mathtools}
\usepackage{pst-plot,pst-node}
\usepackage{mathrsfs,amssymb}
\usepackage{graphicx}
\usepackage{xcolor}
\usepackage[misc]{ifsym}
\theoremstyle{definition}
\newtheorem{definition}{Definition}[section]
\newtheorem{theorem}[definition]{Theorem}

\newtheorem{proposition}[definition]{Proposition}
\newtheorem{Lemma}[definition]{Lemma}

\numberwithin{equation}{section}
\usepackage{hyperref}
\allowdisplaybreaks[4] 

\def\wp{{W^{k,p}(I;X)}}

\begin{document}
\lineskip 1.5em
\begin{center}
\vspace{.5cm}
 {\bf \large Mild solutions of time fractional Navier-Stokes equations driven by  finite delayed external forces}\\
\vspace{.5cm}

          {\bf Md Mansur Alam \footnote{\Letter{gmansuralam94@gmail.com}}},   {\bf Shruti Dubey\footnote{\Letter{sdubey@iitm.ac.in}}}\\
                      Department of Mathematics\\                        Indian Institute of Technology Madras\\
                              Chennai-600 036, India. \\

\end{center}

\begin{abstract}
In this work, we consider time-fractional Navier-Stokes equations (NSE) with the external forces involving finite delay. Equations are considered on a bounded domain $\Omega \subset \mathbb{R}^3$ having sufficiently smooth boundary. We transform the system of equations (NSE) to an abstract Cauchy problem and then   investigate local existence and uniqueness of the mild solutions for the initial datum $\phi \in C\big([-r,0];D(A^\frac{1}{2}) \big)$, where $r>0$ and $A$ is the Stokes operator. With some suitable condition on initial datum we establish the global continuation and regularity of the mild solutions. We use  semigroup theory, tools of fractional calculus and Banach contraction mapping principle to establish our results. 
\end{abstract}
\section{Introduction}
The Navier-Stokes equations (NSE) are the prime system of equations in the study of fluid dynamics which represent the motion of a viscous fluid passing through a region. One may consider the situation when the fluid passes through such a medium that the fluid motion behaves anomalously. To control such system one may consider the external forces having some hereditary features which depends not only on the present state of the system but also on the past history of the system. Therefore, from the last two decades, the  study of NSE with force term consisting of such delay received lot of attention. For instance see \cite{CAR, GAR, GUZ, rubio, pedro} and references therein. On the other hand, the study of  time fractional functional differential equations has gained a huge attention from the researchers, not only due to its novel applications in the field of science and engineering study but also due to the non-local nature of fractional derivatives. In particular, generalized model of a diffusion phenomena in a porous media behaves much better than the classical model of that diffusion phenomena. So, it is significant to consider  time fractional NSE with delay model which reads as follows;

Let $\Omega \subset \mathbb{R}^3$ be a bounded domain with sufficiently smooth boundary $\partial  \Omega$.

\begin{equation}\label{abf1}
\begin{cases}
~^cD^{\alpha}_t u -\Delta u + (u\cdot\nabla)u = -\nabla p+ f(t, u_t), \quad t>0, \\ \nabla \cdot u =0, \\ u|_{\partial \Omega}=0, \\u(x, t) = \phi(x, t), \quad  -r \leq t \leq 0,\quad  x \in \Omega,
\end{cases}
\end{equation}
where $u=(u_1(x, t), u_2(x, t), u_3(x, t))$ represents the velocity of the fluid, $p=p(x, t)$ is the associated pressure,  $[-r, 0]$ is the finite delayed interval, $u_t(\theta)=u(t+\theta)$, $-r \leq \theta \leq 0$,  $f$ is an external force which is given in terms of the past history of the velocity,  $\phi $ is the initial datum corresponding to delayed interval and $~^cD^{\alpha}_t u$ is the Caputo fractional order derivative of order $\alpha \in (0, 1)$.

J. Leray \cite{leray} was the first who has initially  contributed to the mathematical study of NSE. After that Kato-Fujita \cite{KF1, KF2} has proved the existence, uniqueness and regularity of the mild  solutions in space-time variable of the classical NSE  by transforming the system  into an abstract initial value problem and using semigroup theory. From  last few decades there has been lot of work on the study of classical NSE, for instance, see \cite{Giga1985, Giga1983, miyakawa1981, weissler} and references therein. Caraballo et al. \cite{CAR} was the first who considered integer order NSE with finite delay and proved the existence of weak solution in a bounded domain. For similar  investigation of these  problems on unbounded domain and unbounded delay, one may refer  \cite{GAR, GUZ}. In contrast to this, M. El-Shahed et al. \cite{shahed},  was the first who  considered time fractional Navier-Stokes equation and studied the analytical solutions by using Laplace, Fourier and Hankel transformation technique. After that, few more works have been reported on the study of analytical solutions of the similar problem in  \cite{momani, odibat, ganji}. In 2015, Carvalho-Neto et al. \cite{carvalho-neto} have studied about mild solutions to the time-fractional Navier-Stokes equations on $\mathbb{R}^{N}$.  Yong Zhou et al. \cite{zhou} have studied existence, uniqueness and regularity of mild solution for the time fractional NSE without delay on a half-space in  $\mathbb{R}^n, n \geq 3$. Recently Yejuan Wang et al. \cite{wang}  proved the global existence, regularity and decay of mild solution of fractional Navier-Stokes inclusions when the initial velocity belongs to $C([-r, 0];D(A^\epsilon))$, where $0<\epsilon <\frac{1}{2}$, by using some techniques of measure of noncompactness in $L^p$-framework. However, no work has been reported on the analysis of fractional order NSE with delay. Our aim in  this paper is to investigate the existence, uniqueness and regularity of mild solution for fractional order NSE driven by finite delayed forces in $L^2$-framework. 

The paper is organized as follows. In section $2$, we recall some definitions, preliminary results on estimation of analytic solution operators and the nonlinear term $Fu=-P(u\cdot \nabla)u$. In section $3$, we present our main results concerning local existence of mild solution of the problem (\ref{abf1}). In section $4$, we study about the maximality of interval of existence and blow up of the mild solution.  Regularity of mild solution is given in section $5$.

\section{Preliminaries}

This section recalls basic definitions, notations and preliminary results which will be used throughout the paper. We use the standard notations, $\mathbb{R}$, $\mathbb{N}$ for denoting the set of real numbers and natural numbers respectively. Let $X$ be a Banach space with the norm $\lVert . \rVert _X$. For two Banach spaces $X$ and $Y$, $\mathcal{B}(X;Y)$ denotes the space of all bounded linear map from $X$ to $Y$. For $X=Y$, we write $\mathcal{B}(X;X)$ as $\mathcal{B}(X)$. Let $1\leq p< \infty$, then for any interval $I$ in $\mathbb{R}$, $L^p(I; X)$ denotes the set of all $X$-valued measurable functions $f$ on $I$ such that $\int_I \lVert f(t) \rVert_X^p dt< \infty$ and is a Banach space endowed with the norm $\lVert f \rVert=  \big(\int_{I} \lVert f(t) \rVert_X^p \big )^{1/p}$.\\
$\wp=\{f \in L^p(I; X): f  \text{ has weak derivatives } f^{(j)} \text{ and } f^{(j)} \in L^p(I; X) \text{ for all } 1\leq j\leq k , \text{ where } j, k \in \mathbb{N}\}$ is known as  Sobolev spaces of order $k$. It is a Banach space with respect to the norm  $\lVert f \rVert_{1,p}=\sum_{j=0}^{k} \lVert f^{(j)} \rVert$.\\
Let $\Omega \subset \mathbb{R}^3$ be any domain. $W^{k,p}(\Omega)= \{ u \in L^p(\Omega) :  \text{ the weak derivatives } \partial^\gamma u\in L^p(\Omega) \ \forall \text { multi-index } \\
\gamma \text{ such that } |\gamma| \leq k, k \in \mathbb{N}\} $ are standard Sobolev spaces. For $p=2$, $W^{k,p}(\Omega)=H^k(\Omega)$ are Hilbert spaces. Let $C^\infty_0 (\Omega)$ be the set of all infinitely differentiable function with compact support and $H^1_0(\Omega)$ be the closure of $C^\infty_0 (\Omega)$ in $H^1(\Omega)$. \\
We denote $C^k(I;X)$ as the set of all $X$-valued continuously differentiable function upto order $k \in \mathbb{N}$ on $I$. $C^\theta(I;X)$ denotes the set of all H\"older continuous function with H\"older exponent $\theta \in (0,1)$.
 \begin{definition}
 Let $0<\alpha<1$, $a,b \in \mathbb{R}$ and $f\in L^1([a,b]; X)$. The \textit{Riemann Liouville integral} of order $\alpha$ is defined by, $$J^{\alpha}_tf(t)=\frac{1}{\Gamma(\alpha)} \int_{a}^{t} (t-s)^{\alpha-1}f(s) ds \quad \text{a.e }  t \in [a,b]. $$
 \end{definition}
 For $\alpha >0$, we consider \[
  g_\alpha(t) =
  \begin{cases}
    \frac{t^{\alpha-1}}{\Gamma(\alpha)} & t>0, \\
    0 & t\leq 0.
  \end{cases}
\]
 \begin{definition}
 Let $0<\alpha<1$, and $f \in L^1([a,b];X)$ be such that $g_{1-\alpha}*f \in W^{1,1}([a,b];X) $. Then the \textit{Caputo fractional derivative} of order $\alpha$ is defined by, $$~^cD^{\alpha}_t f(t)= \frac{d}{dt} J^{1-\alpha}_t (f(t)-f(a))= \frac{1}{\Gamma(1-\alpha)} \frac{d}{dt}\int_{a}^{t} (t-s)^{-\alpha}(f(s)-f(a)) ds \quad \text{a.e }  t \in [a,b].$$
 \end{definition}
Note that if $f \in C^1([a,b];X)$ then $~^cD^{\alpha}_t f(t)= J^{1-\alpha}_tf'(t)$ for all $t \in [a, b]$.\\
To start with the problem, we transform the system of equations (\ref{abf1}) to an abstract Cauchy problem. Let $\Omega$ be a bounded domain in $\mathbb{R}^3$ with sufficiently smooth boundary $\partial \Omega$ and
$L^2_{\sigma}(\Omega)=\text{closure of }{C^{\infty}_{0,\sigma}(\Omega)}$ in $L^2(\Omega)^3$, where $C^{\infty}_{0,\sigma}(\Omega)=\{u \in C^{\infty}_0(\Omega)^3 : \nabla\cdot u=0 \} $. Then $L^2_{\sigma}(\Omega)$, endowed with the usual inner product in $L^2(\Omega)^3$ is a Hilbert space. To avoid any confusion, we denote the norm on $L^2_{\sigma}(\Omega)$ as $\lVert . \rVert$.\\
Let $G(\Omega)=\{f\in L^2(\Omega)^3: \exists p \in L^2(\Omega) \text{ such that } f=\nabla p \} $. Then $G(\Omega)$ is a closed subspace of  $L^2(\Omega)^3$ and the decomposition $L^2(\Omega)^3=L^2_{\sigma}(\Omega) \oplus G(\Omega) $ holds and known as \textit{Helmholtz decomposition}. Let $P:L^2(\Omega)^3 \rightarrow L^2_{\sigma}(\Omega)$ be the Projection operator.\\
Now we define  the bilinear form as $a(u,v)= <\nabla u, \nabla v>$ where $u, v \in H^1_{0,\sigma} (\Omega)=\overline{C^{\infty}_{0,\sigma}(\Omega)}^{H^1(\Omega)}$, $<,>$ is the usual inner product on $L^2_{\sigma}(\Omega)$ and $A:D(A) \subset L^2_{\sigma}(\Omega) \rightarrow L^2_{\sigma}(\Omega)$ is the associated operator of the bilinear form. Following \cite[Theorem 1.52]{ouhabaz}, $-A$ generates analytic semigroup of contractions $\{T(t)\}_{t\geq 0}$ on $L^2_{\sigma}(\Omega)$. Moreover by following \cite{sohr}, $A=-P\bigtriangleup$ with $D(A)=L^2_{\sigma}(\Omega) \cap H^1_{0}(\Omega) \cap H^2(\Omega)$ is known as Stokes operator. Since $0 \in \rho(A)$ \cite{sohr}, where $\rho(A)$ is the resolvent set of $A$ and $-A$ generates the analytic semigroup $\{T(t)\}_{t\geq0}$, then one can define the fractional power of $A$ as follows \cite{pazy};\\
For $\alpha>0$, $A^{-\alpha}$ is defined by $A^{-\alpha}:= \frac{1}{\Gamma(\alpha)} \int_{0}^{\infty} t^{\alpha-1}T(t)dt$ which is convergent in the uniform operator topology. Also, $A^{-\alpha}$ is injective and hence define $A^\alpha:={(A^{-\alpha})}^{-1}$ with $D(A^{\alpha})= R(A^{-\alpha})$ which is densely defined closed operator in $L^2_{\sigma}(\Omega)$. For $0<\alpha<1$, $A^{-\alpha}$ is bounded linear operator and hence $D(A^\alpha)$ with the norm ${\lVert x \rVert }_{D(A^{\alpha})} =\lVert A^{\alpha}x \rVert $ (which is equivalent to the graph norm on $D(A^\alpha)$ ) is a Banach space and for $0<\alpha<\beta<1$, $D(A^{\beta}) \hookrightarrow D(A^{\alpha}) $.
Applying projection operator on (\ref{abf1}) and using Stokes operator, the system (\ref{abf1}) transforms to the following evolution equation in a Banach space $L^2_{\sigma}(\Omega)$:
\begin{equation} \label{abf2}
\left.\begin{aligned}
~^cD^{\alpha}_t u+A u &=Fu +Pf(t, u_t), \quad t>0, \\ u(t) &= \phi(t), \, \ \ \ \qquad  -r \leq t \leq 0,
\end{aligned}\right\}
\end{equation}
where $Fu=-P(u\cdot\nabla)u$.\\
Let $0<\alpha<1$. Now, we define following two families of operator on $L^2_{\sigma}(\Omega)$,
\begin{align}
\label{def:S}S_{\alpha}(t) &=\frac{1}{2\pi i} \int_{\Gamma_{\varrho,\eta}} e^{{\lambda}t} {\lambda}^{\alpha-1} ({\lambda}^{\alpha}+A)^{-1}d\lambda, \quad t>0,\\
\label{def:T}T_\alpha(t) &= \frac{1}{2 \pi i}\int_{\Gamma_{\varrho,\eta}} e^{{\lambda}t}  ({\lambda}^{\alpha}+A)^{-1} d\lambda, \quad t>0,
\end{align}

where $\Gamma_{\varrho,\eta}$ is a suitable path in $\rho(-A)$. For more details see \cite{guswanto}.\\
Using these operators and some tools of fractional calculus, we define the mild solution of (\ref{abf2}) as follows;
\begin{definition}
Let $0<T<\infty$. A function $u:[-r, T] \rightarrow D(A^{\frac{1}{2}})$ is said to be a local  mild solution of the problem (\ref{abf2})  if $u|_{[0,T]} \in C([0, T]; D(A^{\frac{1}{2}}))$ and $u$ satisfies the following integral equations;
\begin{equation}\label{int.eq}
u(t) =\left \{ \begin{aligned} &S_{\alpha}(t)\phi(0)+ \int_{0}^{t} T_{\alpha}(t-s) Fu(s) ds + \int_{0}^{t} T_{\alpha}(t-s) Pf(s,u_s)ds, \ t \in (0,T],\\
&\phi(t), -r\leq t \leq 0.
\end{aligned} \right.
\end{equation}
\end{definition}
\begin{definition}
Let $0<T<\infty $. A function $u \in C([-r,T];L^2_{\sigma}(\Omega))$ is said to be a classical solution of the problem (\ref{abf2}) if it satisfies following conditions;
\begin{enumerate}[(i)]
\item $u \in C([0,T]; D(A))$,
\item $g_{1-\alpha}*(u-u(0)) \in C^1((0,T);L^2_{\sigma}(\Omega)))$, 
\item $u$ satisfies (\ref{abf2}).
\end{enumerate}
\end{definition}
\begin{Lemma}\cite{guswanto}\label{le:S}
 Let $S_{\alpha}(t)$ be defined by (\ref{def:S}). Then following holds;
\begin{enumerate}[(i)]
\item $S_{\alpha}(t) \in \mathcal{B}(L^2_{\sigma}(\Omega))$ for each $t>0$. Moreover, $\exists$ $C_1=C_1(\alpha)>0$ such that $\lVert S_{\alpha}(t) \rVert \leq C_1$ for all $t>0$.
\item $S_{\alpha}(t) \in \mathcal{B}(L^2_{\sigma}(\Omega);D(A))$ for each $t>0$. Moreover, $\exists$ $C_2=C_2(\alpha)>0$ such that $\lVert AS_{\alpha}(t) \rVert \leq C_2 t^{-\alpha}$ for all $t>0$ and if $x\in D(A)$, then $AS_{\alpha}(t)x=S_{\alpha}(t)Ax$ for all $t>0$.
\item The function $t \mapsto S_{\alpha}(t)$ belongs to $C^\infty((0,\infty);\mathcal{B}(L^2_{\sigma}(\Omega))$ and there exists $M_n=M_n(\alpha)>0$ such that $\lVert S^{(n)}_{\alpha}(t) \rVert \leq M_n t^{-n}$, $t>0$, $n\in \mathbb{N}$.
\item For each $0<\beta<1$, $\exists$ $C_3=C_3(\alpha, \beta)>0$ such that $\lVert A^{\beta} S_{\alpha}(t) x \rVert \leq C_3 t^{-\alpha} (t^{-\alpha({\beta-1})}+1) \lVert x \rVert$, for all $x\in L^2_{\sigma}(\Omega)$, $t>0$.
\item For $x \in L^2_{\sigma}(\Omega)$, $\lim\limits_{t \to 0} \lVert S_{\alpha}(t)x-x \rVert=0$
\end{enumerate}
\end{Lemma}
\begin{Lemma}\cite{guswanto}\label{le:T}
Let $T_\alpha(t)$ be defined by (\ref{def:T}). Then following holds;
\begin{enumerate}[(i)]
\item $T_{\alpha}(t) \in \mathcal{B}(L^2_{\sigma}(\Omega))$ for each $t>0$. Moreover, $\exists$ $B_1=B_1(\alpha)>0$ such that $\lVert T_{\alpha}(t) \rVert \leq B_1 t^{\alpha-1}$ for all $t>0$.
\item $T_{\alpha}(t) \in \mathcal{B}(L^2_{\sigma}(\Omega);D(A))$ for each $t>0$. Moreover, $\exists$ $B_2=B_2(\alpha)>0$ such that $\lVert AT_{\alpha}(t) \rVert \leq B_2 t^{-1}$ for all $t>0$ and if $x\in D(A)$, then $AT_{\alpha}(t)x=T_{\alpha}(t)Ax$ for all $t>0$.
\item The function $t \mapsto T_{\alpha}(t)$ belongs to $C^\infty((0,\infty);\mathcal{B}(L^2_{\sigma}(\Omega))$ and there exists $N_n=N_n(\alpha)>0$ such that $\lVert T^{(n)}_{\alpha}(t) \rVert \leq N_n t^{\alpha-1-n}$, $t>0$, $n\in \mathbb{N}$.
\item For each $0<\beta<1$, $\exists$ $B_3=B_3(\alpha, \beta)>0$ such that $\lVert A^{\beta} T_{\alpha}(t) x \rVert \leq B_3 t^{\alpha(1-\beta)-1} \lVert x \rVert$, for all $x\in L^2_{\sigma}(\Omega)$, $t>0$.
\item For $x \in L^2_{\sigma}(\Omega)$ and $t>0$, $\frac{d}{dt}( S_{\alpha}(t)x)=AT_{\alpha}(t)x$.
\end{enumerate}
\end{Lemma}
\begin{Lemma}\cite{KF2}\label{le:E1}
Let $u, v \in D(A^\frac{1}{2})$, then following estimations hold;
\begin{enumerate}[(i)]
\item There exists $c_1>0$ such that $\lVert A^{-\frac{1}{4}}Fu \rVert \leq c_1 \lVert A^{\frac{1}{2}}u \rVert ^2$,
\item $\lVert A^{-\frac{1}{4}}(Fu-Fv)\rVert \leq c_1 \lVert A^{\frac{1}{2}} (u-v) \rVert (\lVert A^{\frac{1}{2}}u \rVert +  A^{\frac{1}{2}}v \rVert)$.

\end{enumerate}
\end{Lemma}
\begin{Lemma}\cite{KF1}\label{le:E2}
Let $u,v \in D(A^\frac{3}{4})$, then following estimations hold;
\begin{enumerate}[(i)]
\item There exists $c_2>0$ such that $\lVert Fu \rVert \leq c_2 \lVert A^{\frac{1}{2}}u \rVert \lVert A^{\frac{3}{4}}u \rVert $,
\item $\lVert(Fu-Fv)\rVert \leq c_2 (\lVert A^{\frac{1}{2}}(u-v) \rVert \lVert A^{\frac{3}{4}}u \rVert + \lVert A^{\frac{3}{4}}(u-v) \rVert \lVert A^{\frac{1}{2}}v \rVert )$.
\end{enumerate}
\end{Lemma}
\begin{Lemma}\cite[p.~890]{zhou}\label{continuity.es}
Let $0<\beta<1$ and  $T_{\alpha}(t)$ is defined by (\ref{def:T}). Then there exists $B_4=B_4(\alpha,\beta)>0$ such that $\lVert A^{\beta} T_{\alpha}(t) - A^{\beta} T_{\alpha}(s) \rVert \leq B_4 (s^{\alpha(1-\beta)} - t^{\alpha(1-\beta)}) $ for all $s,t>0 \text { with } t>s$. In another words, $t \mapsto A^\beta T(t)$ is continuous for $t>0$ with respect to uniform operator topology.
 \end{Lemma}
 \begin{Lemma}\cite[p.~A3]{luca}\label{int.clsd.op}
  Let $X$ be a Banach space and  $A:D(A)\subset X \rightarrow X$ be a closed operator. Let $I$ be a real interval with inf $I$ $=a$, sup $I$ $=b$, where $-\infty \leq a<b\leq\infty$ and $f:I\rightarrow D(A)$ be such that the functions $t \rightarrow f(t)$, $t\rightarrow Af(t)$ are integrable (Bochner sense) on $I$. Then $$ \int_{a}^{b} f(t) dt \in D(A), \quad A \int_{a}^{b} f(t) dt= \int_{a}^{b} Af(t) dt.$$
 \end{Lemma}
\begin{theorem}\label{contraction.principle} {\textbf{(Contraction Principle)}}
Let $B$ be a closed subset of a Banach space $X$ and $f:B\rightarrow B$ be a contraction map. Then there exists a unique fixed point of $f$ in $B$.
\end{theorem}
\begin{proposition}\label{commutivity}
Let $0<\beta<1$  and $S_{\alpha}(t)$ be defined by (\ref{def:S}) on $L^2_{\sigma}(\Omega)$. Then for $x \in D(A^\beta)$,
\begin{center}
 $A^\beta S_{\alpha}(t) x= S_{\alpha}(t) A^\beta x \ \text{ for all } t>0.$
 \end{center} 
\end{proposition}
\begin{proof}Consider the Mainardi function, $M_{\alpha}(t) := \frac{1}{\pi} \sum\limits_{n=1}^{\infty}  \frac{(-t)^{n-1}}{(n-1)!} \Gamma(\alpha n) \sin(\pi \alpha n), \ $ $t \geq 0$, for more details on Mainardi function, see \cite{mainardi}. Then, following \cite{carvalho}, $S_{\alpha}(t) $ can be written as, $$S_{\alpha}(t)x= \int_{0}^{\infty} M_{\alpha}(s) T(st^\alpha)x ds, \ x \in L^2_{\sigma}(\Omega).$$
Since, $ \int_{0}^{\infty} s^qM_\alpha (s) ds=\frac{\Gamma(q+1)}{\Gamma(\alpha q+1)}$, where $-1<q<\infty$ and $\lVert A^\beta T(t) \rVert \leq t^{-\beta}$ for all $t >0$. Therefore, by Lemma (\ref{int.clsd.op}) we have, $$A^\beta S_{\alpha}(t)x = A^\beta \int_{0}^{\infty} M_{\alpha}(s) T(st^\alpha)x ds=\int_{0}^{\infty} M_{\alpha}(s)A^\beta T(st^\alpha)x ds= \int_{0}^{\infty} M_{\alpha}(s) T(st^\alpha)A^\beta x ds= S_{\alpha}(t) A^\beta x.$$
\end{proof}
\section{Local existence of mild solution}
In this section, we establish local existence and uniqueness of mild solution to (\ref{abf2}).
\begin{theorem}\label{main.result1}
Let $Y_{\frac{1}{2}}:= C([-r, 0]; D(A^{\frac{1}{2}}))$  and $U \subset Y_{\frac{1}{2}}$ be open. Assume that $Pf:[0, \infty) \times U  \rightarrow L^2_{\sigma} (\Omega)$ be  such that,
\begin{enumerate}[(i)]
\item $\lVert Pf(t, \varphi) \rVert \leq \omega(t) \lVert \varphi \rVert _{Y_{\frac{1}{2}}}$ for all $t\geq0$, $\varphi \in U$ and for some $\omega \in L_{loc}^p[0,\infty)$, where $p>\frac{2}{\alpha}$,
\item $\lVert Pf(t, \varphi)-Pf(t, \psi) \rVert \leq  L_f \lVert \varphi-\psi \rVert_{Y_{\frac{1}{2}}}$ for all $\varphi, \psi \in U$ and for some $L_f >0$,
\end{enumerate}
Then for every $\phi \in U$, there exists a unique mild solution $u:[-r, T] \rightarrow D(A^{\frac{1}{2}})$ to (\ref{abf2}), for some $T=T(\phi)>0$.
\end{theorem}
\begin{proof}
Let $\phi \in U$ and $R>0$ be such that $\{\psi \in Y_{\frac{1}{2}} : \lVert \psi - \phi \rVert_{Y_{\frac{1}{2}}} \leq R \} \subset U$.
Let $T>0$ (will be fixed later). We define the following set,
$$Z_{\frac{1}{2}}=\big \{u \in C\big ([-r, T]); D (A^{\frac{1}{2}})\big ) : \ u_0=\phi \text{ and }\lVert u_t -\phi \rVert_{Y_{\frac{1}{2}}} \leq R \text{ for all } t \in [0, T] \big \},$$ which is a non-empty closed subspace of  $ C \big ([-r, T]); D(A^{\frac{1}{2}})\big )$, where $ C \big([-r, T]); D(A^{\frac{1}{2}})\big )$ is endowed with sup-norm topology.
Now we define an operator on $Z_{\frac{1}{2}}$ as follows,
\begin{equation}\label{int.op}
Ku(t) =\left \{ \begin{aligned}
 &S_{\alpha}(t)\phi(0)+ \int_{0}^{t} T_{\alpha}(t-s) Fu(s) ds + \int_{0}^{t} T_{\alpha}(t-s) Pf(s,u_s)ds, \ t \in (0,T], \\ &\phi(t), \ -r\leq t \leq 0.
\end{aligned} \right.
\end{equation}
First we prove that $K(Z_{\frac{1}{2}}) \subset Z_{\frac{1}{2}}$.\\
 Let $u \in Z_{\frac{1}{2}}$. We note that $\lVert u(t) \rVert_{D(A^{\frac{1}{2}})} = \lVert u_t(0)\rVert_{D(A^{\frac{1}{2}})} \leq \lVert u_t \rVert_{Y_{\frac{1}{2}}} \leq R+ \lVert \phi \rVert_{Y_{\frac{1}{2}}}$ for all $t\in[0,T]$. \\
  Let $t_1>0$ be such that $\lVert \phi(t+\theta)- \phi(\theta) \rVert_{D(A^{\frac{1}{2}})}\leq \frac{R}{4}$ for all $t\in[0, t_1]$ and $\theta \in [-r, 0]$ with $t+\theta \leq 0.$
Let $t_2>0$ such that $\lVert S_\alpha(t) \phi(0)-\phi(0) \rVert_{D(A^{\frac{1}{2}})} \leq \frac{R}{4}$  for all $t \in [0, t_2]$. By following \cite[Theorem 2.6]{kai}, $J^{\frac{\alpha}{2}}_t \omega(t) \rightarrow 0$ as $t \rightarrow 0$, choose $t_3>0$ such that $ \int_{0}^{t} (t-s)^{\frac{\alpha}{2}-1} \omega(s)ds \leq \frac{R}{4B_3 (R+\lVert \phi \rVert_{Y_{\frac{1}{2}}})} $ for all $t \in [0,t_3]$.
Also we can choose some $t_4>0$ such that $\int_{0}^{t} (t-s)^{\frac{\alpha}{4}-1} ds\leq \frac{R}{4B_3c_1(R+\lVert \phi \rVert_{Y_{\frac{1}{2}}})^2}$ for all $t \in [0,t_4].$\\
Let $T_1=\text{min} \{t_1,t_2,t_3,t_4 \}$. Now for all $t \in [0,T_1]$ and $\theta \in [-r,0]$ such that $0 \leq t+\theta \leq T_1$, we have 
\begin{align*}
\lVert (Ku)_t(\theta) - \phi(\theta) \rVert_{D(A^\frac{1}{2})} \leq & \lVert  S_{\alpha}(t+\theta) \phi(0)- \phi(0) \rVert_{D(A^{\frac{1}{2}})}+\lVert A^{\frac{1}{2}} \int_{0}^{t+\theta} A^{\frac{1}{4}} T_{\alpha}(t+\theta-s) A^{-\frac{1}{4}} Fu(s) ds \rVert\\
 &+ \lVert \phi(\theta)- \phi (0)\rVert_{D(A^{\frac{1}{2}})} + \lVert A^{\frac{1}{2}} \int_{0}^{t+\theta} T_{\alpha}(t+\theta-s) Pf(s,u_s)ds \rVert  \\
\leq & \lVert  S_{\alpha}(t+\theta) \phi(0)- \phi(0) \rVert_{D(A^{\frac{1}{2}})} +B_3  \int_{0}^{t+\theta} (t+\theta-s)^{\frac{\alpha}{4}-1}  \lVert A^{-\frac{1}{4}} Fu(s)\rVert ds \\
 &+ \lVert \phi(\theta)- \phi (0)\rVert_{D(A^{\frac{1}{2}})}+  B_3 \int_{0}^{t+\theta} (t+\theta-s)^{\frac{\alpha}{2}-1} \omega(s) \lVert u_s \rVert_{Y_{\frac{1}{2}}}ds \\
 \leq & \lVert  S_{\alpha}(t+\theta) \phi(0)- \phi(0) \rVert_{D(A^{\frac{1}{2}})} +B_3c_1  \int_{0}^{t+\theta} (t+\theta-s)^{\frac{\alpha}{4}-1}  \lVert u(s)\rVert_{D(A^{\frac{1}{2}})}^2 ds \\
 &+ \lVert \phi(\theta)- \phi (0)\rVert_{D(A^{\frac{1}{2}})}+  B_3 \int_{0}^{t+\theta} (t+\theta-s)^{\frac{\alpha}{2}-1} \omega(s) \lVert u_s \rVert_{Y_{\frac{1}{2}}}ds
 \\
 \leq & \lVert  S_{\alpha}(t+\theta) \phi(0)- \phi(0) \rVert_{D(A^{\frac{1}{2}})} +B_3c_1  \int_{0}^{t+\theta} (t+\theta-s)^{\frac{\alpha}{4}-1}  ( R+ \lVert \phi \rVert_{Y_{\frac{1}{2}}})^2 ds \\
 &+ \lVert \phi(\theta)- \phi (0)\rVert_{D(A^{\frac{1}{2}})}+  B_3 \int_{0}^{t+\theta} (t+\theta-s)^{\frac{\alpha}{2}-1} \omega(s)  (R+ \lVert \phi \rVert_{Y_{\frac{1}{2}}})ds\\
 \leq & R
\end{align*} 
Hence $\lVert (Ku)_t -\phi \rVert_{Y_{\frac{1}{2}}} \leq R$ for all $t \in [0, T_1]$. Now we prove the continuity of $t \mapsto Ku(t)$ on $(0,T_1]$ with respect to the topology induced by  $D(A^{\frac{1}{2}})$-norm.\\
First define $v(t):= \int_{0}^{t} T_{\alpha}(t-s) Fu(s) ds$ and let $t_0 \in(0,T_1] $ with $t>t_0$ and $\delta >0$ small enough.
\begin{align}
\nonumber \lVert A^\frac{1}{2} (v(t)-v(t_0)) \rVert \leq &\lVert A^{\frac{1}{2}} \int_{0}^{t_0-\delta} A^{\frac{1}{4}} [T_{\alpha}(t-s)-T_{\alpha}(t_0-s)] A^{-\frac{1}{4}}Fu(s) ds \rVert \\
\nonumber &+ \lVert A^{\frac{1}{2}} \int_{t_0-\delta}^{t_0} A^{\frac{1}{4}} [T_{\alpha}(t-s)-T_{\alpha}(t_0-s)] A^{-\frac{1}{4}}Fu(s) ds \rVert 
 \\
\nonumber &+ \lVert A^{\frac{1}{2}} \int_{t_0}^{t} A^{\frac{1}{4}} T_{\alpha}(t-s) A^{-\frac{1}{4}}Fu(s) ds \rVert   \\
\nonumber := & I_1+I_2+I_3
\end{align}
Consider $I_1$. We see that,
\begin{align}
 \nonumber I_1 \leq & c_1 \underset{0\leq s \leq t_0-\delta} \sup \lVert A^{\frac{3}{4}}[T_{\alpha}(t-s)-T_{\alpha}(t_0-s)]\rVert \int_{0}^{t_0-\delta} \lVert A^{\frac{1}{2}}u(s)\rVert^2 ds \\ 
\nonumber \leq &c_1 \underset{0\leq s \leq t_0-\delta} \sup \lVert A^{\frac{3}{4}}[T_{\alpha}(t-s)-T_{\alpha}(t_0-s)]\rVert (R+\lVert \phi \rVert_{Y_\frac{1}{2}})^2 (t_0-\delta)
 \end{align}
Since by Lemma (\ref{continuity.es}) $t \mapsto A^{\frac{3}{4}}T(t)$ is continuous in the uniform operator topology on $[\delta, T_1]$ for every $\delta>0$, there exists $\tilde{t} \in [0, t_0-\delta]$ such that, $$\underset{0\leq s \leq t_0-\delta} \sup \lVert A^{\frac{3}{4}}[T_{\alpha}(t-s)-T_{\alpha}(t_0-s)]\rVert = \lVert A^{\frac{3}{4}}[T_{\alpha}(t-\tilde{t})-T_{\alpha}(t_0-\tilde{t})]\rVert \rightarrow 0, \text{ as } t \rightarrow t_0$$ and hence $I_1 \rightarrow 0$ as $t \rightarrow t_0$.\\
Now consider $I_2$. Using Lemmas (\ref{le:T})(iv), (\ref{le:E1})(i) we have,
\begin{align*}
I_2 &\leq \int_{t_0-\delta}^{t_0 }\lVert A^{\frac{3}{4}}[T_{\alpha}(t-s)-T_{\alpha}(t_0-s)] A^{-\frac{1}{4}} Fu(s)\rVert ds \\
&\leq c_1 B_3 \int_{t_0-\delta}^{t_0} [(t-s)^{\frac{\alpha}{4}-1}+(t_0-s)^{\frac{\alpha}{4}-1}] \lVert A^{\frac{1}{2}}u(s)\rVert^2 ds \\
&\leq 2c_1B_3 (R+\lVert \phi \rVert_{Y_\frac{1}{2}})^2 \int_{t_0-\delta}^{t_0} (t_0-s)^{\frac{\alpha}{4}-1}ds \rightarrow 0\text{ as } \delta \rightarrow 0.
\end{align*}
Again using Lemmas (\ref{le:T})(iv), (\ref{le:E1})(i) in $I_3$ we have,
\begin{align*}
I_3 &\leq \int_{t_0}^{t} \lVert A^{\frac{3}{4}} T_{\alpha}(t-s) A^{-\frac{1}{4}}Fu(s)\rVert ds \\
&\leq c_1B_3\int_{t_0}^{t} (t-s)^{\frac{\alpha}{4}-1} \lVert A^{\frac{1}{2}}u(s)\rVert^2 ds\\
& \leq (R+\lVert \phi \rVert_{Y_\frac{1}{2}})^2 c_1B_3\int_{t_0}^{t} (t-s)^{\frac{\alpha}{4}-1}ds  \rightarrow 0 \text { as } t \rightarrow t_0.
\end{align*}
Therefore $\lVert A^\frac{1}{2} (v(t)-v(t_0)) \rVert \rightarrow 0 $ as $t \rightarrow t_0+$. Analogously it can be proved that $\lVert A^\frac{1}{2} (v(t)-v(t_0)) \rVert \rightarrow 0 $ as $t \rightarrow t_0-$ by considering $t<t_0$. Hence $t \mapsto v(t)$ is continuous on $(0, T_1]$ with respect to  the topology induced by $D(A^{\frac{1}{2}})$-norm.\\
Now define $w(t):= \int_{0}^{t} T_{\alpha}(t-s) Pf(s,u_s) ds$ and let $t_0 \in(0,T_1] $ with $t>t_0$ and $\delta >0$ small enough.
\begin{align}
\nonumber \lVert A^\frac{1}{2} (w(t)-w(t_0)) \rVert \leq &\lVert A^{\frac{1}{2}} \int_{0}^{t_0-\delta}  [T_{\alpha}(t-s)-T_{\alpha}(t_0-s)] Pf(s,u_s)ds \rVert \\
\nonumber &+ \lVert A^{\frac{1}{2}} \int_{t_0-\delta}^{t_0} [T_{\alpha}(t-s)-T_{\alpha}(t_0-s)] Pf(s,u_s) ds \rVert 
 \\
\nonumber &+ \lVert A^{\frac{1}{2}} \int_{t_0}^{t}  T_{\alpha}(t-s) Pf(s,u_s) ds \rVert   \\
\nonumber := & J_1+J_2+J_3
\end{align}
Consider $J_1$. We see that,
\begin{align}
 \nonumber J_1 \leq &  \underset{0\leq s \leq t_0-\delta} \sup \lVert A^{\frac{1}{2}}[T_{\alpha}(t-s)-T_{\alpha}(t_0-s)]\rVert \int_{0}^{t_0-\delta}  \omega(s)\lVert u_s \rVert_{Y_{\frac{1}{2}}} ds \\ 
\nonumber \leq & \underset{0\leq s \leq t_0-\delta} \sup \lVert A^{\frac{1}{2}}[T_{\alpha}(t-s)-T_{\alpha}(t_0-s)] \rVert \int_{0}^{t_0-\delta}  \omega(s) (R+\lVert \phi\rVert_{Y_\frac{1}{2}}) ds
 \end{align}
Since for any $b>0$,  $\omega \in L^p[0,b]$, therefore $\omega \in L^1[0,b]$. Also since by Lemma (\ref{continuity.es}) $t \mapsto A^{\frac{1}{2}}T(t)$ is continuous in the uniform operator topology on $[\delta, T_1]$ for every $\delta>0$,  so there exists $\tau \in [0, t_0-\delta]$ such that, $$\underset{0\leq s \leq t_0-\delta} \sup \lVert A^{\frac{1}{2}}[T_{\alpha}(t-s)-T_{\alpha}(t_0-s)]\rVert = \lVert A^{\frac{1}{2}}[T_{\alpha}(t-\tau)-T_{\alpha}(t_0-\tau)] \rVert\rightarrow 0 \text{ as } t \rightarrow t_0$$ and hence $J_1 \rightarrow 0$ as $t \rightarrow t_0$.\\
Similarly considering $J_2$ and using Lemma (\ref{le:T})(iv) we have,
\begin{align}
\nonumber J_2 &\leq \int_{t_0-\delta}^{t_0 }\lVert A^{\frac{1}{2}}[T_{\alpha}(t-s)-T_{\alpha}(t_0-s)] \rVert \omega(s)\lVert u_s\rVert_{Y_{\frac{1}{2}}} ds \\
\nonumber &\leq B_3 \int_{t_0-\delta}^{t_0} [(t-s)^{\frac{\alpha}{2}-1}+(t_0-s)^{\frac{\alpha}{2}-1}] \omega(s) (R+\lVert \phi \rVert_{Y_\frac{1}{2}}) ds \\
&\label{J2}\leq 2B_3 (R+ \lVert \phi \rVert_{Y_{\frac{1}{2}}}) \int_{t_0-\delta}^{t_0} (t_0-s)^{\frac{\alpha}{2}-1} \omega(s) ds.
\end{align}
Since for any $b>0$,  $\omega \in L^p[0,b]$ with $p>\frac{2}{\alpha}$, therefore R.H.S of (\ref{J2}) $\rightarrow 0 $ as $ \delta \rightarrow 0$.\\
Again by using Lemma (\ref{le:T})(iv) in $J_3$ we have,
\begin{align}
\nonumber J_3 &\leq \int_{t_0}^{t} \lVert A^{\frac{1}{2}} T_{\alpha}(t-s) \omega(s)\lVert u_s\rVert_{Y_{\frac{1}{2}}} ds \\
\nonumber &\leq B_3\int_{t_0}^{t} (t-s)^{\frac{\alpha}{2}-1}  \omega(s) (R+\lVert \phi \rVert_{Y_\frac{1}{2}}) ds\\
&\label{J3}\leq B_3( R+ \lVert \phi \rVert_{Y_{\frac{1}{2}}} )\int_{t_0}^{t} (t-s)^{\frac{\alpha}{2}-1} \omega(s)ds  \rightarrow 0 \text { as } t \rightarrow t_0.
\end{align}
Again, since for any $b>0$,  $\omega \in L^p[0,b]$ with $p>\frac{2}{\alpha}$, therefore R.H.S of (\ref{J3}) $\rightarrow 0 $ as $t \rightarrow t_0$.\\
Therefore $\lVert A^\frac{1}{2} (w(t)-w(t_0)) \rVert \rightarrow 0 $ as $t \rightarrow t_0+$. Analogously it can be proved that $\lVert A^\frac{1}{2} (w(t)-w(t_0)) \rVert \rightarrow 0 $ as $t \rightarrow t_0-$ by considering $t<t_0.$ Hence $t \mapsto w(t)$ is continuous on $(0, T_1]$ with respect to the  topology induced by $D(A^{\frac{1}{2}})$-norm.\\
Since $u(0)=\phi(0) \in D(A^{\frac{1}{2}})$, therefore by Proposition (\ref{commutivity}), Lemma (\ref{le:S})(iii) we can say that $\lVert A^{\frac{1}{2}} (S_{\alpha}(t)\phi(0)-S_{\alpha}(t_0)\phi(0))\rVert= \lVert  S_{\alpha}(t)A^{\frac{1}{2}}\phi(0)-S_{\alpha}(t_0)A^{\frac{1}{2}}\phi(0)\rVert \rightarrow 0$ as $t\rightarrow t_0$.\\
Thus we proved that $t\mapsto Ku(t)$ is continuous on $[-r, T_1]$ with respect the topology induced by $D(A^{\frac{1}{2}})$-norm and hence $K(Z_{\frac{1}{2}}) \subset Z_{\frac{1}{2}}$.\\
Now let $u$, $v \in Z_{\frac{1}{2}}$, $t \in [0, T_1]$. Then using Lemmas (\ref{le:T})(iv), (\ref{le:E1})(ii) we get,
\begin{align}
\nonumber   \lVert A^{\frac{1}{2}} \{ Ku(t)-Kv(t) \} \rVert
\nonumber  & \leq \lVert A^{\frac{1}{2}}\int_{0}^{t} A^{\frac{1}{4}} T_{\alpha}(t-s)  A^{-\frac{1}{4}}(Fu(s)-Fv(s))ds \rVert \\
\nonumber &  +\lVert A^{\frac{1}{2}} \int_{0}^{t}  T_{\alpha}(t-s)  \{P f(s,u_s) -P f(s,v_s)\} ds \rVert\\
\nonumber & \leq B_3\int_{0}^{t} (t-s)^{\frac{\alpha}{4}-1} \lVert  A^{-\frac{1}{4}}(Fu(s)-Fv(s)) \rVert ds \\
\nonumber & + B_3 \int_{0}^{t} (t-s)^{\frac{\alpha}{2}-1} \lVert Pf(s,u_s) -P f(s,v_s) \rVert ds \\
\nonumber & \leq c_1 B_3 \int_{0}^{t} (t-s)^{\frac{\alpha}{4}-1}  \{ \lVert A^{\frac{1}{2}}(u(s)-v(s))\rVert \} \{\lVert A^{\frac{1}{2}}u(s)\rVert +\lVert A^{\frac{1}{2}} v(s) \rVert \}ds \\
\nonumber &+ B_3 L_f \int_{0}^{t} (t-s)^{\frac{\alpha}{2}-1} \lVert u_s-v_s\rVert_{Y_{\frac{1}{2}}} ds \\
\nonumber &  \leq c_1 B_3 \int_{0}^{t} (t-s)^{\frac{\alpha}{4}-1}  \lVert u-v\rVert_{Z_{\frac{1}{2}}} 2 (R+\lVert \phi \rVert_{Y_\frac{1}{2}})ds \\
\nonumber & +B_3 L_f \int_{0}^{t} (t-s)^{\frac{\alpha}{2}-1} \underset{0\leq r \leq s}\sup \lVert A^{\frac{1}{2}} ( u(r)-v(r) )\rVert ds  \\
\nonumber & \leq 2(R+\lVert \phi\rVert_{Y_{\frac{1}{2}}})c_1 B_3 \lVert u-v \rVert_{Z_{\frac{1}{2}}} \int_{0}^{t} (t-s)^{\frac{\alpha}{4}-1} ds\\
&\nonumber  + B_3 L_f \lVert u-v \rVert_{Z_{\frac{1}{2}}} \int_{0}^{t} (t-s)^{\frac{\alpha}{2}-1} ds \\
\label{contraction} & = \big \{ 2(R+\lVert \phi \rVert_{Y_{\frac{1}{2}}})c_1B_3  \int_{0}^{t} (t-s)^{\frac{\alpha}{4}-1} ds + B_3 L_f  \int_{0}^{t} (t-s)^{\frac{\alpha}{2}-1} ds \big \} ||u-v||_{Z_{\frac{1}{2}}} 
\end{align}
Since both the integrals in  R.H.S of (\ref{contraction}) tend to zero as $t\rightarrow 0$, we can choose a small positive $T(\leq T_1)$ such that following holds,
\begin{equation}
\lVert A^{\frac{1}{2}} \{ Ku(t)-Kv(t) \}\rVert \leq M \lVert u-v\rVert_{Z_{\frac{1}{2}}} \text{for all} \ t \in[0, T] \
\text{and some} \ 0<M<1,
\end{equation}
This implies that
 $\lVert Ku-Kv \rVert_{Z_{\frac{1}{2}}} \leq M \lVert u-v \rVert_{Z_{\frac{1}{2}}} \text{ for some } 0<M<1.$\\
Therefore, $K:Z_{\frac{1}{2}}\rightarrow Z_{\frac{1}{2}}$ is a contraction map. Consequently, by contraction mapping principle (\ref{contraction.principle}), $K$ has a unique fixed point $u \in Z_{\frac{1}{2}}$ which satisfies the integral equation (\ref{int.op}). This proves the existence of uniqueness local mild solution of (\ref{abf2}).
\end{proof} 

\section{Continuation of mild solution}
\begin{theorem}\label{main.res.2}
Assume that all the  conditions of the Theorem (\ref{main.result1}) hold for $U=Y_{\frac{1}{2}}$. Then for every $\phi \in Y_{\frac{1}{2}}$,  problem (\ref{abf2}) has a unique mild solution on a maximal interval of existence $[-r, t_{max})$. Moreover if $t_{max} < \infty$ then $\overline{\text{lim}}_{t \to t_{max}} \lVert u(t) \rVert_{D(A^{\frac{1}{2}})}=\infty$.
\end{theorem}
\begin{proof}
From the previous result, we know that the mild solution of (\ref{abf2}) exists in the interval $[-r, T]$. Now we prove that this solution can be extended to the interval $[-r, T+\delta]$ for some $\delta>0$.\\
Let $u$ be the mild solution of (\ref{abf2}) on $[-r, T]$. Define $v(t)=u(t+T)$ where $v(t)$ is a mild solution of
\begin{equation}\label{abf.extension}
\left.\begin{aligned}
~^cD^{\alpha}_t v+A v(t) &=Fv(t)+Pf(t, v_t), \quad t>0,  \\ v_0=u_T,
\end{aligned}\right\}
\end{equation}
Since $u \in C\big([-r,T];D(A^{\frac{1}{2}})\big)$, therefore $v_0=u_T \in Y_{\frac{1}{2}}$. Hence the existence of the mild  solution of (\ref{abf.extension}) on some interval $[-r, \delta]$, where $\delta>0$,  is assured by the Theorem (\ref{main.result1}). 
Consequently, let $[-r,t_{max})$ be the maximal interval of existence of mild solution of (\ref{abf2}).\\
If $t_{max}=\infty$ then the mild solution is global.
If $t_{max}<\infty$ we prove that  $\overline{\text{lim}}_{t \to t_{max}} \lVert u(t) \rVert_{D(A^{\frac{1}{2}})}=\infty$.\\
Let us  assume that $\overline{\text{lim}}_{t \to t_{max}} \lVert u(t) \rVert_{D(A^{\frac{1}{2}})} < \infty$. Consequently, $\overline{\text{lim}}_{t \to t_{max}} \lVert u_t \rVert_{Y_\frac{1}{2}} <\infty$. Then,  there exists $N>0$ such that such that $\lVert u_t \rVert_{Y_\frac{1}{2}} \leq N$ for all $t \in [0, t_{max}).$ This implies $\lVert u(t)\rVert _{D(A^\frac{1}{2})} \leq N$ and $\lVert Pf(t,u_t)\rVert \leq N \omega(t)$ for all $t \in [0,t_{max}).$\\
Let $0 < t<\tau< t_{max}$ and $\delta>0$ be sufficiently small. Then we have
\begin{align*}
\lVert u(t)-u(\tau) \rVert_{D(A^\frac{1}{2})} &= \lVert S_\alpha(t)\phi(0) -S_\alpha(\tau) \phi(0)\rVert_{D(A^\frac{1}{2})}+\lVert \int_{t}^{\tau} T_\alpha(\tau-s) Pf(s,u_s)ds \rVert_{D(A^\frac{1}{2})}\\
&\quad+\lVert \int_{0}^{t-\delta} [T_\alpha(t-s)-T_\alpha(\tau-s)] Pf(s,u_s)ds \rVert_{D(A^\frac{1}{2})}\\
&\quad+\lVert \int_{t-\delta}^{t} [T_\alpha(t-s)-T_\alpha(\tau-s)]Pf(s,u_s)ds \rVert_{D(A^\frac{1}{2})}\\
&\quad+\lVert \int_{0}^{t-\delta} [T_\alpha(t-s)-T_\alpha(\tau-s)] Fu(s)ds \rVert_{D(A^\frac{1}{2})}\\
&\quad+\lVert \int_{t-\delta}^{t} [T_\alpha(t-s)-T_\alpha(\tau-s)] Fu(s)ds \rVert_{D(A^\frac{1}{2})}+\lVert \int_{t}^{\tau} T_\alpha(\tau-s) Fu(s)ds \rVert_{D(A^\frac{1}{2})}\displaybreak[1] \\
&\leq \lVert S_\alpha(t)\phi(0) -S_\alpha(\tau) \phi(0)\rVert_{D(A^\frac{1}{2})}+N B_3 \int_{t}^{\tau} (\tau-s)^{\frac{\alpha}{2}-1} \omega(s)ds \\
&\quad +N \underset{0\leq s \leq t-\delta} \sup \lVert A^{\frac{1}{2}}[T_{\alpha}(t-s)-T_{\alpha}(\tau-s)]\rVert \int_{0}^{t-\delta} \omega(s)ds\\
&\quad +2N\int_{t-\delta}^{t} (t-s)^{\frac{\alpha}{2}-1} \omega(s) ds+N^2 t_{max} \underset{0\leq s \leq t-\delta} \sup \lVert A^{\frac{3}{4}}[T_{\alpha}(t-s)-T_{\alpha}(\tau-s)]\rVert \\
&\quad + 2N \int_{t-\delta}^{t} (t-s)^{\frac{\alpha}{4}-1} ds +N^2 \int_{t}^{\tau} (\tau-s)^{\frac{\alpha}{4}-1}ds
\end{align*}
Since $\omega \in L_p[0, t_{max})$ for $p>\frac{2}{\alpha}$, therefore by applying H\"older's inequality in 2nd and 4th integrals of the above inequality and using the fact (\ref{continuity.es}), it is easy to check that R.H.S of the above inequality can be made arbitrarily small by  choosing $|t-\tau|$ sufficiently small. Hence $t \mapsto u(t)$ is uniformly continuous on $(0, t_{max})$ with respect to the topology induced by $D(A^\frac{1}{2})$-norm. This implies that $\lim\limits_{t \to t_{max} }u(t)=u(t_{\max})$ exists, which contradicts the maximality of the interval of existence. So our assumption is wrong. Hence the theorem is proved.
\end{proof}

\section{Regularity result}
In this section, we prove the regularity of the mild solution of the problem (\ref{abf2}). If we could  prove that the function $t \mapsto Fu(t)+Pf(t, u_t)$ is H\"older continuous on the interval $[0,T]$ in a Banach space $L^2_{\sigma}(\Omega)$, then the mild solution of (\ref{abf2}) is classical one \cite{li}. But we found that for the mild solution $u :[-r,T] \rightarrow D(A^\frac{1}{2})$ of (\ref{abf2}), the H\"older continuity of $t \mapsto Fu(t)$ can not be proved in  $L^2_{\sigma}(\Omega)$. To overcome this difficulty, we choose  initial datum $\phi$ such that it belongs to the space $U$ which is an open subset of $ Y_{\frac{3}{4}} :=C([-r,0];D(A^{\frac{3}{4}}))$. Further, we consider the following assumptions;
\begin{enumerate}[(I)]
\item $\lVert Pf(t, \varphi) \rVert \leq \omega(t) \lVert \varphi \rVert _{Y_{\frac{3}{4}}}$ for all $t\geq0$, $\varphi \in U$ and for some $\omega \in L_{loc}^p[0,\infty)$, where $p>\frac{4}{\alpha}$,
\item$\lVert Pf(t, \varphi)-Pf(t, \psi) \rVert \leq  L_f \lVert \varphi-\psi \rVert_{Y_{\frac{3}{4}}}$ for all $\varphi, \psi \in U$ and for some $L_f >0$,
\end{enumerate}
Then analogous to the proof of the Theorem (\ref{main.result1}), it can be proved that under the above assumptions (I), (II) and Lemma (\ref{le:E2}), there exists unique local mild solutions $u \in C([-r,T];D(A^{\frac{3}{4}}))$ of (\ref{abf2}) such that $\lVert u_t -\phi \rVert_{Y_{\frac{3}{4}}} \leq R$  for all $t \in [0,T]$, for some $R>0$.\\
Now we prove the regularity of this mild solution in the following theorem.
\begin{theorem}\label{Regularity}
Let $u \in C([-r,T];D(A^{\frac{3}{4}}))$ be the local mild solution of the evolution system (\ref{abf2}) such that $\lVert u_t -\phi \rVert_{Y_{\frac{3}{4}}} \leq R$  for all $t \in [0,T]$ and for some $R>0$. Also we assume the following hypotheses;
\begin{enumerate}[(H1)]
\item $\phi \in C^\theta ([-r,0];D(A^{\frac{3}{4}}))$ such that $\phi(0) \in D(A)$, for some $\theta \in (0,1)$.\item $Pf$ be such that $\lVert Pf(t, \varphi)- Pf(s, \psi) \rVert \leq L (|t-s|^\theta+\lVert \varphi- \psi \rVert_{ Y_{\frac{3}{4}}})$ for all $t,s \in [0,T]$ and $\varphi, \psi \in Y_{\frac{3}{4}}$, for some $\theta \in (0,1)$. 
\end{enumerate}
Then the mild solution is a classical solution.
\end{theorem}
To prove the above theorem we first need to prove the following results;
\begin{Lemma}\label{hold1}
Let $0<\beta<1$ and define $v(t)= \int_{0}^{t} T_{\alpha}(t-s) Pf(s, u_s) ds$, $t \geq 0$, where $u \in C([-r,T];D(A^{\frac{3}{4}}))$. Then $v(t) \in D(A^\beta)$ for all $t \in [0,T]$. Moreover $A^\beta v \in C^{\alpha(1-\beta)}([0,T];L^2_{\sigma}(\Omega))$.
\begin{proof}
Since $u \in C([-r,T];D(A^{\frac{3}{4}}))$, the map  $t \mapsto u_t$ is continuous  on $[0, T]$ with respect to $D(A^{\frac{3}{4}}))$ norm. By assumption (H2), $t \mapsto Pf(t,u_t)$ is continuous on $[0,T]$ with respect to $L^2_{\sigma}(\Omega)$ norm. So there exists $N>0$ such that $\lVert Pf(t,u_t)  \rVert \leq N,$ for all $t \in [0,T]$.

Now, by using Lemma (\ref{le:T})(iv) we see that $ \lVert A^\beta T_{\alpha}(t-s) Pf(t,u_t) \rVert \leq N B_3 (t-s)^{\alpha(1-\beta)-1}$ which is integrable on $(0,t)$ and since $A^\beta$ is closed operator, by Lemma (\ref{int.clsd.op}) $$\lVert A^\beta v(t) \rVert \leq N B_3 \int_{0}^{t} (t-s)^{\alpha(1-\beta)-1} ds = \frac{NB_3}{\alpha (1-\beta)} t^{\alpha(1-\beta)} \leq \frac{N B_3}{\alpha(1-\beta)} T^{\alpha(1-\beta)} <\infty \ \text{ for all }  t \in [0,T].$$
Let $t \in [0,T]$ and $h>0$ such that $t+h \in [0,T]$. Without loss of generality assume $0<h<1$.
\begin{align*}
A^\beta v(t+h)- A^\beta v(t) &= A^\beta \int_{0}^{t+h} T_{\alpha}(t+h-s) Pf(s,u_s) ds-A^\beta \int_{0}^{t} T_{\alpha}(t-s) Pf(s,u_s) ds\\
&=A^\beta \int_{0}^{t}[T_{\alpha}(t+h-s)-T_{\alpha}(t-s)] Pf(s, u_s)ds+\int_{t}^{t+h} T_{\alpha}(t+h-s)Pf(s,u_s)ds \\
&:= v_1+v_2
\end{align*}
Now, by Lemma (\ref{continuity.es})
\begin{align*}
\lVert v_1 \rVert &\leq NB_4\int_{0}^{t} [(t-s)^{\alpha(1-\beta)-1}-(t+h-s)^{\alpha(1-\beta)-1}] ds \\
&= \frac{NB_4}{\alpha(1-\beta)} [t^{\alpha(1-\beta)}-(t+h)^{\alpha(1-\beta)}+h^{\alpha(1-\beta)}]\\
&\leq \frac{NB_4}{\alpha(1-\beta)} h^{\alpha(1-\beta)}
 \end{align*}
 Also, using Lemma (\ref{le:T})(iv) we have
 \begin{align*}
 \lVert v_2\rVert &\leq N B_3 \int_{t}^{t+h} (t+h-s)^{\alpha(1-\beta)-1} ds\\
 &= \frac{NB_3}{\alpha(1-\beta)} h^{\alpha(1-\beta)}
 \end{align*}
 Hence $A^\beta v \in C^{\alpha(1-\beta)}([0,T];L^2_{\sigma}(\Omega)) $.
\end{proof}
\end{Lemma}
\begin{Lemma}\label{hold2}
Let $0<\beta<1$ and define $w(t)= \int_{0}^{t} T_{\alpha}(t-s) Fu(s) ds$, $t \geq 0$, where $u \in C([-r,T];D(A^{\frac{3}{4}}))$. Then $w(t) \in D(A^\beta)$ for all $t \in [0,T]$. Moreover $A^\beta w \in C^{\alpha(1-\beta)}([0,T];L^2_{\sigma}(\Omega))$.
\end{Lemma}
\begin{proof}
According to  the condition $\lVert u_t -\phi \rVert_{Y_{\frac{3}{4}}} \leq R$  for all $t \in [0,T]$, Lemma (\ref{le:E2})(i) and using the property $D(A^\frac{3}{4}) \hookrightarrow D(A^\frac{1}{2})$, it is easy to check that $t \mapsto Fu(t)$ is bounded on $[0,T]$. Then, by following the  similar arguments as in the proof of Lemma (\ref{hold1}), we can show that $A^\beta w \in C^{\alpha(1-\beta)}([0,T];L^2_{\sigma}(\Omega))$.
\end{proof}
\begin{Lemma}\label{hold3}
Let $0<\beta<1$ and $x \in D(A)$. Consider $\Psi (t)= S_{\alpha}(t)x$, $t \geq0$. Then the map $t \mapsto \Psi(t)$ is H\"older continuous on $[0, T]$ with respect to $D(A^\beta)$-norm.
\end{Lemma}
\begin{proof}
Let $t \in [0,T]$ and $h>0$ such that $t+h \in [0,T]$. Without loss of generality assume $0<h<1$.\\\
By Lemma (\ref{le:T})(ii), (iv)  and (v) we have
\begin{align*}
\lVert A^\beta \Psi(t+h) -A^\beta \Psi (t) \rVert &= \lVert A^\beta S_{\alpha}(t+h)x- A^\beta S_{\alpha}(t)x \rVert \\
& = \lVert A^\beta \int_{t}^{t+h} \frac{d}{d\tau}(S_{\alpha}(\tau)x) d\tau \rVert\\
&= \lVert  A^\beta \int_{t}^{t+h} AT_{\alpha}(\tau) x d\tau \rVert \\
&= \lVert  A^\beta \int_{t}^{t+h} T_{\alpha}(\tau) Ax d\tau \rVert \\
& \leq B_3 \int_{t}^{t+h} \tau^{\alpha(1-\beta)-1} \lVert Ax \rVert d\tau \\
&= \frac{B_3 \lVert Ax \rVert}{\alpha(1-\beta)} [(t+h)^{\alpha(1-\beta)}-t^{\alpha(1-\beta)}]\\
& \leq \frac{B_3 \lVert Ax \rVert}{\alpha(1-\beta)} h^{\alpha (1-\beta)}
\end{align*}
Hence $\Psi \in C^{\alpha(1-\beta)}([0,T]; D(A^\beta))$.
\end{proof}
\begin{proof}[Proof of the Theorem (\ref{Regularity})]
If $u  \in C([-r,T];D(A^{\frac{3}{4}}))$ is the mild solution of the Cauchy problem (\ref{abf2}), then $u(t)= \Psi(t)+v(t)+w(t)$ for all $t \in [0,T]$. Therefore by Lemmas (\ref{hold1}), (\ref{hold2}), (\ref{hold3}), the map $t \mapsto u(t)$ is H\"older continuous on $[0,T]$ with respect to $D(A^\frac{3}{4})$-norm.\\
According to the condition $\lVert u_t -\phi \rVert_{Y_{\frac{3}{4}}} \leq R$  for all $t \in [0,T]$, the estimation in Lemma (\ref{le:E2})(ii) and using the property $D(A^\frac{3}{4}) \hookrightarrow D(A^\frac{1}{2})$, it can be proved that $t \mapsto Fu(t)$ is H\"older continuous on $[0,T]$ with respect to $L^2_{\sigma}(\Omega)$-norm. \\
Now we prove that $t \mapsto Pf(t, u_t)$ is H\"older continuous on $[0, T]$ in $L^2_{\sigma}(\Omega)$.\\
Let $t\in[0,T]$ and $h>0$ be such that $t+h \in [0,T].$ Without loss of generality assume $0<h<\text{min}\{1,r\}$.
\begin{align*}
\lVert u_{t+h}-u_t \rVert_{Y_{\frac{3}{4}}} & =\underset{-r\leq z \leq 0}\sup \lVert u(t+h+z) - u(t+z) \rVert_{D(A^{\frac{3}{4}})}  \\
&\leq \underset{-r\leq \tau \leq t}\sup \lVert u(h+\tau) - u(\tau) \rVert_{D(A^{\frac{3}{4}})}\\
&\leq  \underset{-r\leq \tau \leq -h}\sup \lVert u(h+\tau) - u(\tau) \rVert_{D(A^{\frac{3}{4}})} + \underset{-h\leq \tau \leq 0}\sup \lVert u(h+\tau) - u(\tau) \rVert_{D(A^{\frac{3}{4}})}\\
&\quad + \underset{0\leq \tau \leq t}\sup \lVert u(h+\tau) - u(\tau) \rVert_{D(A^{\frac{3}{4}})}\\
& \leq \underset{-r\leq \tau \leq -h}\sup \lVert \phi(h+\tau) - \phi(\tau) \rVert_{D(A^{\frac{3}{4}})} + \underset{-h\leq \tau \leq 0}\sup \lVert u(h+\tau) - \phi(0) \rVert_{D(A^{\frac{3}{4}})} \\
&\quad + \underset{-h\leq \tau \leq 0}\sup \lVert \phi(0) - \phi(\tau) \rVert_{D(A^{\frac{3}{4}})} + \underset{0\leq \tau \leq t}\sup \lVert u(h+\tau) - u(\tau) \rVert_{D(A^{\frac{3}{4}})}
\end{align*}
Since $\phi \in C^\theta ([-r,0];D(A^{\frac{3}{4}}))$ and $u \in C^{\alpha(1-\beta)}([0,T];D(A^\frac{3}{4}))$, therefore we have 
\begin{center}
$\underset{-r\leq \tau \leq -h}\sup \lVert \phi(h+\tau) - \phi(\tau) \rVert_{D(A^{\frac{3}{4}})} \leq L_1 h^\theta,$\\
$\underset{-h\leq \tau \leq 0}\sup \lVert u(h+\tau) - \phi(0) \rVert_{D(A^{\frac{3}{4}})} \leq L_2 (h+\tau)^{\alpha(1-\beta)}\leq L_2 h^{\alpha(1-\beta)}$,\\
$\underset{-h\leq \tau \leq 0}\sup \lVert \phi(0) - \phi(\tau) \rVert_{D(A^{\frac{3}{4}})} \leq L_1 (-\tau)^\theta \leq L_1 h^\theta$,\\
$\underset{0\leq \tau \leq t}\sup \lVert u(h+\tau) - u(\tau) \rVert_{D(A^{\frac{3}{4}})}\leq L_2 h^{\alpha(1-\beta)}$,
\end{center}
where, $L_1$, $L_2$, $L_3$, $L_4$ are positive constants.
This shows that $t \mapsto u_t$ is H\"older continuous on $[0,T]$ in $Y_{\frac{3}{4}}$. Hence by assumption (H2), the map $t \mapsto Pf(t,u_t)$ is H\"older continuous on $[0,T]$ in a Banach space $L^2_{\sigma}(\Omega)$.\\
Thus, it is proved that the mild solution is a classical solution.  
\end{proof}
\section*{Conclusion}
 The existence of $D(A^{\frac{1}{2}})$-valued local mild solution has been established for a time-fractional NSE driven by finite delayed external forces by using Banach fixed point theorem when the initial datum belong to an open subset $U$ of $D(A^{\frac{1}{2}})$. It is also proved that local mild solution can be continued globally if the initial datum curve belong to the whole space $U=D(A^\frac{1}{2})$.   Regularity result has been demonstrated by considering more stronger initial datum curve and suitable assumption on forces.
\section*{Acknowledgement}
 The authors acknowledge the support provided by Ministry of Human Resource Development(MHRD), Government of India. 
\bibliographystyle{abbrv}
\bibliography{REF}
\end{document}